\documentclass[10pt]{amsart}
\usepackage[all]{xy}
\usepackage{graphicx}
\usepackage[utf8]{inputenc}
\usepackage{amssymb}

\newcommand{\R}{{\mathbb R} }

\newcommand{\cO}{{\mathcal O} }

\newcommand{\cS}{{\mathcal S} }
\newcommand{\cT}{{\mathcal T} }

\newcommand{\cX}{{\mathcal X} }

\newcommand{\cH}{{\mathcal H} }
\newcommand{\cK}{{\mathcal K} }

\newcommand{\wt}{\widetilde}
\newcommand{\pt}{\partial}
\def\ol#1{{\overline{#1}}}
\newtheorem*{Maintheorem*}{Main Theorem}
\newtheorem*{theorem*}{Theorem}

\newtheorem*{conjecture*}{Conjecture}

\newtheorem{theorem}{Theorem}
\newtheorem{definition}{Definition}
\newtheorem{lemma}{Lemma}
\newtheorem{remark}{Remark}
\newtheorem*{remark*}{Remark}
\newtheorem{proposition}{Proposition}
\newtheorem{corollary}{Corollary}

\def\ke{K{\"a}h\-ler-Ein\-stein }
\def\ks{Ko\-dai\-ra-Spen\-cer }
\def\ka{K{\"a}h\-ler}

\def\wp{Weil-Pe\-ters\-son }

\newcommand{\ii}{\mbox{\footnotesize$\sqrt{-1}$}}
\def\ddb{\sqrt{-1}\partial\overline{\partial}}

\def\cinf{C^\infty}

\def\gab{{g_{\alpha\ol\beta}}}

\def\db{{\ol\partial}}

\def\we{\wedge}

\def\CY{Calabi-Yau }

\def\equov#1{\hskip1.5mm\vtop{\hbox{$=$}\vskip-4mm\hbox{\tiny$\!\eqref{#1}$}}\hskip1.5mm}
\def\gequov#1{\hskip1.5mm\vtop{\hbox{$\geq$}\vskip-1.8mm\hbox{\tiny$\!\eqref{#1}$}}\hskip1.5mm}

\title{Kähler forms for families of Calabi-Yau manifolds}

\author[M.~Braun]{Matthias Braun}
\address{Fachbereich Mathematik und Informatik,
Philipps-Universität Marburg, Lahnberge, Hans-Meerwein-Straße, D-35032
Marburg, Germany}
\email{braunm@mathematik.uni-marburg.de}
\author[Y.-J.~Choi]{Young-Jun Choi}
\email{youngjun.choi@pusan.ac.kr}
\address{Department of Mathematics, Pusan National University, 2, Busandaehak-ro 63beongil, Geumjeong-gu, Busan, 46241, Korea}
\author[G.~Schumacher]{Georg Schumacher}
\email{schumac@mathematik.uni-marburg.de}
\address{Fachbereich Mathematik und Informatik,
Philipps-Universität Marburg, Lahnberge, Hans-Meerwein-Straße, D-35032
Marburg, Germany}


\begin{document}

\begin{abstract}
\ke metrics for polarized families $(f:\cX \to S, \lambda_{\cX/S})$ of Calabi-Yau manifolds define a natural hermitian metric on the relative canonical bundle $K_{\cX/S}$. The computation of the curvature form being equal to the pull-back of the \wp form up to a numerical constant is used for the construction of a \ka\ form on $\cX$, whose restriction to the fibers is Ricci flat.
\end{abstract}

\maketitle

\section{Introduction}
S.T.~Yau's solution of the Calabi problem (\cite{yau, calyau}) and later developments had a strong impact on the theory of moduli spaces. Recalling the work of Griffiths \cite{gri} one could see that for a holomorphic family of polarized Calabi-Yau manifolds (with trivial canonical bundle) the curvature form of the Hodge bundle on the base space is equal to the \wp form.

Generalizing moduli of Calabi-Yau manifolds and canonically polarized varieties, in  \cite{f-s:extremal} moduli of extremal \ka\ manifolds were constructed and studied.
In this note we want to focus on the curvature of the relative canonical bundle on the total space of a holomorphic family of \CY manifolds -- somewhat analogous to the situation of (effective) families of canonically polarized manifolds, equipped with \ke metrics, where the relative canonical bundle on the total space is strictly positive (\cite{sch:inv12,sch-preprint08}).

Let $(f:\cX \to S, \lambda_{\cX/S})$ be a holomorphic, polarized family of Calabi-Yau manifolds $\cX_s= f^{-1}(s)$, i.e.\ compact manifolds with $c_{1,\R}(\cX_s)=0$, equipped with Ricci flat \ka\ forms $\omega_{\cX_s}$. The relative volume form $\omega^n_{\cX/S}= g\, dV$ induces a hermitian metric $g^{-1}$ on the relative canonical bundle $K_{\cX/S}$. We arrive at the following fact.

\medskip
\noindent\textit{
The curvature form of $\cK_{\cX/S}$ is equal to the pull-back of the \wp form up to a numerical constant:
\begin{equation}\label{eq:relcan}
 2\pi c_1(K_{\cX/S}, g^{-1})= \ddb \log g = \frac{1}{ {\rm vol}\,\cX_s} f^* \omega^{WP}.
\end{equation}
}
The above statement is related to the computation of the \wp form as curvature form of the Hodge metric on the base of a family, (where the canonical bundles of the fibers are assumed to be trivial). We will give a short direct argument in terms of the Ricci flat metrics, which will also be used in the proof of the following theorem. The results are related to \cite{br:diss} and to \cite{bcs}.

Given a reduced complex space $\cX$, we denote by $H^{1,1}(\cX)_\R$ the $\pt\ol\pt$-cohomology. It consist of classes of locally $\pt\ol\pt$-exact differentiable $(1,1)$-forms modulo $\pt\ol\pt$-exact forms (cf.\ \cite[Section 1]{f-s:extremal}, also for the treatment of singularities). This cohomology theory is suitable for our purpose.

\begin{theorem}\label{th}
  Let $(f:\cX \to S, \lambda_{\cX/S})$ be a holomorphic, polarized, effective family of Calabi-Yau manifolds $\{\cX_s\}_{s\in S}$, where $\cX_s  = f^{-1}(s)$ such that
  \begin{itemize}
  \item[(a)] the polarization $\lambda_{\cX/S} \in R^1f_*\Omega^1_{\cX/S}(S)$ is represented by a class $\lambda_\cX\in H^{1,1}(\cX)_\R$ on  $\cX$,

      \noindent or

  \item[(b)] the first Betti-numbers of the fibers vanish: $b_1(\cX_s)=0.
    $
  \end{itemize}
  Assume that the Green's functions for functions on the fibers $\cX_s$ are uniformly bounded from below (by a negative constant). Then there exists a \ka\ form $\wt\omega_\cX$ on $\cX$, whose restrictions to the fibers $\cX_s$ are the Ricci flat forms on $(\cX_s, \lambda_{\cX_s})$.
\end{theorem}
The \ka\ form $\wt\omega_\cX$ will be equal to $\omega_\cX + c \cdot f^*\omega^{WP}$ for some $c>0$, and $\omega_\cX$ given in Proposition~\ref{pr:omegaX} below.

In particular the assumption on the existence of $\lambda_\cX$ is satisfied if the total space $\cX$ is a compact \ka\ manifold (with the \ka\ class inducing the polarization). The effectiveness of the family is needed so that $\wt\omega_\cX$ is positive definite rather than positive semi-definite.

Concerning the existence of such an estimate for the Green's function required in  Theorem~\ref{th} the following is known. It holds, if the diameter of the fibers is uniformly bounded, or if a uniform isoperimetric inequality holds by the work or J.~Cheeger \cite{ch}, Cheeger-Yau \cite{c-y}, Chr.~Croke \cite{cr} and others. (Note that the volume of the fibers is constant).

For families of Calabi-Yau manifolds the latter questions were investigated recently by Sh.~Takayama \cite{ta}, and by X.~Rong and Y.~Zhang \cite{r-z}, who give a positive answer for the projective case, furthermore in \cite{zhang} for moduli of polarized Calabi-Yau manifolds, and  V.~Tosatti \cite{tosatti16}.

\section{Polarized families of \ka\ manifolds}
Let $X$ be a compact \ka\ manifold and $\lambda_X\in  H^1(X,\Omega^1_X)_\R\subset  H^2(X, \R)$ a \ka\ class. We summarize some facts about deformations of polarized \ka\ manifolds (cf.\ \cite{f-s:extremal,Sch1}).

A {\em holomorphic family of compact polarized manifolds} $\{(\cX_s,\lambda_{\cX_s})\}_{s\in S}$, parameterized by a reduced complex space $S$, is  given by a proper, holomorphic submersion $f :\cX \to S$ such that $\cX_s = f^{-1}(s)$ for $s\in S$ together with a section $\lambda_{\cX/S} \in R^1f_* \Omega^1_{\cX/S}(S)$, whose restrictions $\lambda_{\cX/S}|\cX_s$ are equal to the polarizations $\lambda_{\cX_s}$. In \cite{f-s:extremal} we showed that a polarization $\lambda_{\cX/S}$ can be described in an alternative way requiring that $\lambda_{\cX/S} \in R^2f_* \R(S)$ with the property that again the restrictions to fibers $\cX_s$ are the given \ka\ classes.

At least locally, with respect to the base, polarizations for holomorphic families can be represented by \ka\ forms on the total space. After replacing $S$ by a neighborhood of a given point, any polarization $\lambda_{\cX/S}$ can be represented by a \ka\ form $\omega_\cX$. (For singular spaces, by definition, the existence of local differentiable $\ddb$-potentials for \ka\ forms is always required).

A {\em deformation} of a compact polarized manifold $(X,\lambda_X)$ over a complex space $(S, s_0)$ with a distinguished point $s_0\in S$ consists of a holomorphic family $(f:\cX \to S, \lambda_{\cX/S})$ together with an isomorphism $ X \to \cX_{s_0}$ that takes $\lambda_{\cX_{s_0}}$ to $\lambda_X$.

{\em Infinitesimal deformations} are characterized as follows. Let $\cT_X$ be the sheaf of holomorphic vector fields on $X$.
The \ka\ class $\lambda_X\in H^1(X, Hom_{\cO_X}(\cT_X, \cO_X))$ defines the equivalence class of an extension, given by the {\em Atiyah sequence}
\begin{equation*}\label{eq:atiyah}
0 \to \cO_X \to \Sigma_X \to \cT_X \to 0.
\end{equation*}
We note that the edge homomorphisms of the induced cohomology sequence are defined by the cup product with the polarization:
$$
\cup\lambda_X : H^q(X, \cT_X) \to H^{q+1}(X,\cO_X).
$$
The kernel of the map $\cup \lambda_X: H^1(X,\cT_X) \to H^2(X,\cO_X)$  can be identified with the space $H^1(X,\cT_X)_{\lambda_X}$ of infinitesimal deformations of the polarized manifold $(X,\lambda_X)$ (cf.\ Section~\ref{se:curv}).

\section{Deformations of \CY manifolds}
Various definitions of \CY manifolds are common, we will use the following most general definition.
\begin{definition}
  A compact \ka\ manifold $X$ with vanishing first real Chern class is called \CY manifold.
\end{definition}
According to Yau's theorem, for any polarized \CY manifold $(X,\lambda_X)$ there exists a unique Ricci flat \ka\ form $\omega_X$ in the \ka\ class $\lambda_X$.

All complex spaces will be reduced unless stated otherwise. We do not assume that the total space $\cX$ of a family of \CY manifolds is \ka. We first give a necessary condition for the family of \ke metrics to be induced by a $(1,1)$-form on the total space, which is also sufficient for the main theorem.

\begin{proposition}[cf.\ {\cite[Prop.~3.6]{f-s:extremal}}]\label{pr:omegaX}
\begin{itemize}
\item[(a)]
Let $(f:\cX \to S, \lambda_{\cX/S})$ be a polarized family of \CY manifolds of dimension $n$, and let $\omega_{\cX_s}\in \lambda_{\cX_s}$ be the \ke forms. Then locally with respect to $S$ there exists a $d$-closed, real $(1,1)$-form $\omega_\cX$ on the total space $X$ such that
\begin{equation}\label{eq:restr}
\omega_\cX|\cX_s = \omega_{\cX_s}.
\end{equation}
\item[(b)]
If the polarization  can be represented by an element $\wt \lambda\in H^{1,1}(\cX)_\R\to H^2(\cX,\R)\to R^2f_*\R(S)$  on the total space, then $\omega_\cX$ can be chosen globally.
\item[(c)]
Such a global form $\omega_\cX$ in the class of $\wt\lambda$ exists, satisfying the additional equation
\begin{equation}\label{eq:norm}
\int_{\cX/S} \omega^{n+1}_\cX =0,
\end{equation}
by which it is uniquely determined.
\item[(d)]
The assumption in {\rm(b)} is not needed, if the first Betti numbers $b_1(\cX_s)$ vanish.
\end{itemize}
\end{proposition}
\begin{proof}
The local statement (a) is contained in \cite{f-s:extremal}, the statement (b) about the global existence of $\omega_\cX$ can be seen as follows: Let the polarization $\lambda_{\cX/S}$ be represented by a locally $\pt\ol\pt$-exact, real $(1,1)$-form $\tau$. In the notation of \cite{f-s:extremal} the form $\tau$ is a section of $\Phi_{\cX}$. The restrictions $\tau|\cX_s$ represent the unique \ke forms $\omega_{\cX_s}$ so that $\omega_{\cX_s}=\tau|\cX_s + \ddb \varphi_s $ for $\cinf$ functions $\varphi_s$ on the fibers, whose fiberwise harmonic projections vanish so that these depend upon $s\in S$ in a $\cinf$ way. (For singular base spaces one can use the implicit function theorem in the sense of \cite[Section 6]{f-s:extremal}) in order to verify differentiable dependence upon the parameter. Let $\varphi$ denote the resulting function on $\cX$. We set $\omega_\cX= \tau + \ddb \varphi$.

Condition \eqref{eq:norm} can be achieved like in \cite[Prop.~3.6]{f-s:extremal}, where also the uniqueness is shown.

The statement (d) follows from \cite[[Cor.~3.7]{f-s:extremal}.
\end{proof}
Locally $\pt\ol\pt$-exact forms $\omega_\cX$ satisfying \eqref{eq:restr} (for not necessarily smooth base spaces $S$) had been introduced as {\em admissible} forms in \cite[Section 3]{f-s:extremal}, where the existence of such forms was being studied. Because of the Tian-Todorov theorem \cite{tian,todorov} about the unobstructedness of infinitesimal deformations we can restrict ourselves to smooth base spaces, when we compute curvatures and related tensors.

We will use the following notation. Local holomorphic coordinates on $X$ are denoted by $(z^1,\ldots,z^n)$ (with Greek indices), whereas local coordinates on the base $S$ are denoted by $(s^1,\ldots,s^k)$ with Latin indices.

We will write
$$
\omega_{\cX_s}= \ii g_{\alpha\ol\beta}(z,s)dz^\alpha\we dz^\ol\beta
$$
for the \ke forms of vanishing Ricci curvature representing the \ka\ classes $\lambda_{\cX_s}$, and
$$
\omega_\cX = \ii\left(g_{\alpha\ol\beta} dz^\alpha\we dz^{\ol\beta} + g_{i\ol\beta} ds^i\we dz^{\ol\beta} + g_{\alpha\ol\jmath} dz^\alpha\we ds^{\ol\jmath} + g_{i\ol\jmath} ds^i\we ds^{\ol\jmath}  \right).
$$
The $k$-th power of a differential form by definition is equal to the $k$-th exterior product, divided by $k!$.

\subsection{Metric characterization of the \ks map}\label{se:mcharwp} Let $v$ be a  differentiable section of $\cT_\cX \to f^*\cT_S$ for
$$
0 \to \cT_{\cX/S}\to \cT_\cX \stackrel{\stackrel{v}{\longleftarrow}}{\longrightarrow} f^*\cT_S \to 0.
$$
Then, given $s_0\in S$, the exterior derivative $\db\/ v$ defines a homomorphism from $T_{s_0}S$ to the space of $\db$-closed $(0,1)$-forms with values in the sheaf $\cT_{\cX_{s_0}}$. In this way the \ks map is defined in terms of Dolbeault cohomology:
$$
\xymatrix{T_{s_0}S \ar@{}[d]   \ar[r]^{\rho_{s_0}} & H^1(X,\cT_X)\\
\pt/\pt s^i \ar@{}[u]|{\rotatebox{90}{$\in$}} \ar@{|->}[r]& [ \db(v(\pt/\pt s^i))|\cX_{s_0}]\;.\ar@{}[u]|{\rotatebox{90}{$\in$}}
}
$$
Given the situation of Proposition~\ref{pr:omegaX}, let $\omega_\cX$ be a real, closed $(1,1)$-form, whose restrictions to all fibers are the Ricci flat \ka\ metrics $\omega_{\cX_s}$ representing the polarizations $\lambda_{\cX_s}$. As the statement of \eqref{eq:relcan} does not depend upon the global existence of $\omega_\cX$, we do not have to assume that $\cX$ possesses a global \ka\ form.

A particular lift of a tangent vector $\pt/\pt s^i$ to $\cX$ is the {\em horizontal lift}
\begin{equation}\label{eq:vi}
v_i = \pt/\pt s^i + a^\alpha_i \pt/\pt z^\alpha,
\end{equation}
which is characterized by the condition that $v_i$ is perpendicular to the respective fiber with respect to $\omega_\cX$.

Set $(g^{\ol\beta\alpha}) = (g_{\alpha\ol\beta})^{-1}$ and $g= \det(g_{\alpha\ol\beta})$. Furthermore, we will use the notation $\pt_i=\pt/\pt s^i$ and $\pt_\alpha=\pt/\pt z^\alpha$ etc.

Then
\begin{equation}\label{eq:horli}
a^\alpha_i= -g^{\ol\beta\alpha}g_{i\ol\beta}.
\end{equation}
Distinguished representatives of the \ks classes $\rho_{s_0}(\pt/\pt s^i)$ are
$$
A_i=A^\alpha_{i \ol\beta} \pt_\alpha dz^\ol\beta = (\db v_i) | \cX_{s_0}, \text{ i.e. } A^\alpha_{i\ol\beta}= \pt_{\ol\beta}(a^\alpha_i),
$$
where the horizontal lift $v_i$ is given by \eqref{eq:vi}.

We will need covariant derivatives {\em in fiber direction} and use the semi-colon notation $\mbox{\textvisiblespace}_{; \alpha}$ etc. For ordinary derivatives, in particular derivatives with respect to the parameter we will use the $|$-symbol like $\mbox{\textvisiblespace}_{|i}$. Also raising and lowering of indices is being used with respect to the given metrics on the fibers. We recall the following known facts.

\begin{proposition}\label{pr:harmrep}
We identify $X$ with $\cX_{s_0}$. Then the distinguished \ks forms $A_i$ satisfy the following properties.
\begin{eqnarray}
  A_i\cup \omega_{X} &=&0 \text{ i.e. } A_{i\ol\beta\ol\delta}= A_{i\ol\delta\ol\beta}\label{eq:symm} \\
   \db A_i &=&0 \text{ i.e. }
  A^\alpha_{i\ol\beta;\ol\delta}=  A^\alpha_{i\ol\delta;\ol\beta}\label{eq:clsd} \\
  \db^{\,*}\!\!A_i&=&0  \text{ i.e. } - g^{\ol\beta\gamma} A^\alpha_{i\ol\beta;\gamma}=0 \label{eq:dbclsd}
\end{eqnarray}
\end{proposition}
\begin{proof}
  We have $A_i \cup \omega_X = A_{i\ol\beta\ol\delta}\,dz^\ol\beta\we dz^\ol\delta$, where $A_{i\ol\beta\ol\delta}= a_{i\ol\beta;\ol\delta}=- g_{i\ol\beta;\ol\delta}$, which implies the first two equations \eqref{eq:symm} and \eqref{eq:clsd}. The $\ol\pt^*$\!-closedness \eqref{eq:dbclsd} will be shown below.
\end{proof}
Note that \eqref{eq:symm} implies that the class of $A_i$ is contained in $H^1(X, \cT_{X})_{\lambda_{X}}$.

\subsection{Tensors on \CY manifolds}
We will need various properties for tensors on \CY manifolds, some of which are already mentioned in E.~Calabi's original paper \cite{calabi}.
\begin{lemma}[{\cite{calabi}}]\label{le:par}
  Any holomorphic vector field, any holomorphic $1$-form, and any general holomorphic tensor are parallel.
\end{lemma}
\begin{proof}
  Let $\xi^\alpha \pt_\alpha$ be a holomorphic vector field. Then
  $$
  \int_X \xi_{\ol\beta;\gamma}\ol\xi_{\alpha;\ol\delta}g^{\ol\beta\alpha}g^{\ol\delta\gamma} g dV =
  -\int_X  \xi_{\ol\beta;\gamma\ol\delta}\ol\xi_{\alpha}g^{\ol\beta\alpha}g^{\ol\delta\gamma} g dV =
  -\int_X g^{\ol\beta\alpha}\left(g^{\ol\delta\gamma} \xi_{\ol\beta;\ol\delta\gamma} - \xi_\ol\tau R^\ol\tau_{\;\; \ol\beta}\right)\xi_\alpha g\,dV=0
  $$
  because of the holomorphicity of $\xi^\alpha$ and the Ricci flatness of the metric. So $\xi_{\ol\beta;\gamma}=0$. The rest follows in the same way.
\end{proof}

\subsection{Curvature form of the relative canonical bundle}\label{se:curv}
Let a holomorphic family $f:\cX \to S$ of \CY manifolds be given together with a form  $\omega_\cX$ with the properties of Proposition~\ref{pr:omegaX}.

Let $\Theta$ be the curvature form for the relative canonical bundle. Because of the \ke condition for $\omega_{\cX_s}$ the curvature form vanishes when restricted to any fiber:
\begin{eqnarray*}
\Theta = \ddb \log g &=& \ii\left(\Theta_{i\ol\jmath}\, ds^i \we ds^\ol\jmath + \Theta_{i\ol\beta}\, ds^i\we dz^\ol\beta + \Theta_{\alpha\ol\jmath}\, dz^\alpha\we ds^\ol\jmath)\right)\\
\Theta_{\alpha\ol\beta} & = & 0
\end{eqnarray*}
Now
$
\Theta_{\alpha\ol\jmath;\ol\beta}=\Theta_{\alpha\ol\beta|\ol\jmath}=0.
$
This proves
\begin{lemma}\label{le:theta}
The forms $\Theta_{\alpha\ol\jmath}\, dz^\alpha$ are holomorphic on the fibers $\cX_s$
and
$
\Theta_{i\ol\beta;\alpha}=0.
$
\end{lemma}
By Lemma~\ref{le:par}
\begin{eqnarray}
  \Theta_{\alpha\ol\jmath;\gamma}  &=& 0 \label{eq:thag} \\
  \Theta_{i\ol\beta;\ol\delta} &=& 0 \label{eq:thbd}
\end{eqnarray}
We note a consequence of Lemma~\ref{le:par}: Again $X$ denotes a fiber $\cX_s$. The cup product $\cup \lambda_X: H^0(X,\cT_X) \to H^1(X, \cO_X)$ can be represented by the cup product with $\omega_X$ that maps holomorphic vector fields to harmonic $(0,1)$-forms. This map is an isomorphism. In particular we have an exact sequence
$$
0 \to H^1(X, \Sigma_X) \to H^1(X,\cT_X) \stackrel{\cup \lambda_X}{\longrightarrow} H^2(X,\cO_X),
$$
which implies that the space of infinitesimal deformations of the polarized manifold $(X,\lambda_X)$, namely the kernel of the cup product with the polarization on $H^1(X,\cT_X)$ with values in $H^2(X,\cO_X)$, can be identified with $H^1(X, \Sigma_X)$.
\begin{lemma}\label{le:ATh}
  \begin{equation}
    \ol\pt^{*}\! (A^\alpha_{i\ol\beta}\, \pt_\alpha dz^{\ol\beta}) = \Theta_{i\ol\beta}\,g^{\ol\beta\alpha}\,\pt_\alpha.
  \end{equation}
\end{lemma}
\begin{proof}
  \begin{gather*}
  -g^{\ol\beta\gamma}A^\alpha_{i\ol\beta;\gamma}=-g^{\ol\delta\alpha}g^{\ol\beta\gamma} A_{i\ol\delta\ol\beta;\gamma} \equov{eq:symm}-g^{\ol\delta\alpha}g^{\ol\beta\gamma} A_{i\ol\beta\ol\delta;\gamma} \equov{eq:horli} g^{\ol\delta\alpha}g^{\ol\beta\gamma} g_{i\ol\beta;\ol\delta\gamma}=\\
   g^{\ol\delta\alpha}\left(g^{\ol\beta\gamma} g_{i\ol\beta;\gamma\ol\delta}  + g_{i\ol\tau} R^\ol\tau_{\;\ol \delta } \right)=  g^{\ol\delta\alpha}(g^{\ol\beta\gamma} g_{\gamma\ol\beta|i})_{;\ol\delta} =  g^{\ol\delta\alpha} (\log g )_{|i\ol\delta} = g^{\ol\delta\alpha}\Theta_{i\ol\delta}.
  \end{gather*}
\end{proof}
Now we consider the hermitian product $\langle \mbox{\textvisiblespace},\mbox{\textvisiblespace}\rangle_\Theta$ on $\cT_\cX$ that is defined by $\Theta$, (which is not positive definite).

Let
\begin{equation}\label{eq:chi}
\chi_{i\ol\jmath}:= \langle v_i, v_j\rangle_\Theta = \Theta_{i\ol\jmath} - a^\alpha_i \Theta_{\alpha\ol\jmath} - \Theta_{i\ol\beta} a^\ol\beta_{\ol\jmath}  .
\end{equation}
We compute the (fiberwise) Laplacian of $\chi_{i\ol \jmath}$, where $\Box$ denotes the Laplacian for functions with respect to $\omega_{\cX_s}$.
\begin{proposition}
\begin{equation}\label{eq:laplchi}
  \Box \chi_{i\ol\jmath}= - 2 g^{\ol\beta\alpha} \Theta_{i\ol\beta}\Theta_{\alpha\ol\jmath}.
\end{equation}
\end{proposition}
\begin{proof}
Because of Lemma~\ref{le:theta}, and \eqref{eq:thag}, \eqref{eq:thbd}, we see that
\begin{gather*}
  \Box \chi_{i\ol\jmath}= - g^{\ol\delta\gamma} \chi_{i\ol\jmath;\ol\delta\gamma}= - g^{\ol\delta\gamma}\left(\Theta_{i\ol\jmath,\ol\delta\gamma} -  a^\alpha_{i;\ol\delta\gamma} \Theta_{\alpha\ol\jmath} - \Theta_{i\ol\beta} a^\ol\beta_{\ol\jmath;\ol\delta\gamma}  \right)\\
  \hskip3.3cm=-g^{\ol\delta\gamma}\left(\Theta_{\gamma\ol\delta|i\ol\jmath} -A^\alpha_{i\ol\delta;\gamma}\Theta_{\alpha\ol\jmath} - \Theta_{i\ol\beta}A^\ol\beta_{\ol\jmath\gamma;\ol\delta}\right),
\end{gather*}
where the last equality follows from the fiberwise Ricci flatness of the metric. Now the vanishing of the Ricci forms and Lemma~\ref{le:ATh} imply the claim.
\end{proof}
Integration of \eqref{eq:laplchi} over the fibers implies that the forms $\Theta_{\alpha\ol\jmath}\,dz^\alpha$, and their conjugates vanish identically. Now by \eqref{eq:chi} the Laplacians  $\Box\chi_{i\ol\jmath}$ and $\Box\Theta_{i\ol\jmath}$ are equal and vanish. This means that the coefficients $\Theta_{i\ol\jmath}$ only depend on the parameter $s$. Altogether we have the following Corollary.
\begin{corollary}
  \begin{eqnarray}
    \Theta_{i\ol\beta}&=&0 \label{eq:thib} \\
    \Theta_{\alpha\ol\jmath}&=&0\\
    \Theta_{i\ol\jmath}(z,s)&=& \Theta_{i\ol\jmath}(s). \label{eq:Thetaij}
  \end{eqnarray}
\end{corollary}

\begin{proof}[Proof of Proposition~\ref{pr:harmrep} \eqref{eq:dbclsd}]
  The statement follows from Lemma~\ref{le:ATh} together with \eqref{eq:thib}.
\end{proof}

\section{Construction of the global \ka\ form}
We first compute the curvature of the relative canonical bundle for a holomorphic, polarized family of \CY manifolds, equipped with \ke metrics according to Yau's theorem.

In case the canonical bundles of the fibers are trivial, rather than some power of it, the Hodge bundles $\cH^{n,0}$ on the base space of a holomorphic family are provided with the natural metric. It was shown in this case (cf.\ \cite{wang,lu,lu-sun}) that the curvature form of the Hodge line bundle is equal to the \wp form. This fact was based upon the formula of Griffiths \cite{gri} for the curvature of the Hodge bundles over the period domain, combined with G.~Tian's result that characterized the \wp metric for \CY manifolds in terms of the period map \cite{tian}. The analogous result, relating the \wp form to the invariant metrics on period domains for families of holomorphic symplectic manifolds was shown in \cite{Sch0}). The result \eqref{eq:relcan} is somewhat related. However, it applies to the curvature of the relative canonical bundle on the total space of a holomorphic family. For fibers with trivial canonical bundles it is contained the article \cite{wang} by C.L.~Wang. Our approach to \eqref{eq:relcan} in the general case will be needed in the proof of our main theorem.

\subsection{\wp metric for \CY manifolds}
Let $(f:\cX\to S, \lambda_{\cX/S})$ be a holomorphic family of polarized \CY manifolds. For any $s\in S$ we consider the \ks map
$$
\rho_s: T_sS \to H^1(\cX_s,\cT_{\cX_s}).
$$
According to Section~\ref{se:mcharwp},
$$
(\db v_i)|\cX_s=A^\alpha_{i\ol\beta}\pt_\alpha dz^\ol\beta
$$
is the harmonic representative of
$$
\rho_s(\pt/\pt s^i).
$$
The $L^2$ inner product of harmonic representatives is known to define a \ka\ form $\omega^{WP}$ (cf.\ \cite{nanni,siu,f-s:extremal}).  We use the following notation:
\begin{eqnarray*}
  \omega^{WP}&=& \ii G^{WP}_{i\ol\jmath}(s)\, ds^i \we ds^\ol\jmath  \text{, where }\\
   G^{WP}_{i\ol\jmath}(s)&=& \int_{\cX_s} A^{\alpha}_{i\ol\beta} \, A^\ol\delta_{\ol\jmath \gamma}\, g_{\alpha\ol\delta}\, g^{\ol\beta\gamma} g\, dV.
\end{eqnarray*}
The above formulas are valid for smooth base spaces $S$, or if these are singular, the local coordinates $s_i$ are taken on a local smooth, ambient space $U$, with a local minimal embedding $S\subset U$. Since the Kuranishi space is known to be smooth, and the construction of \wp metric is functorial, i.e.\ compatible with base change, we can restrict ourselves to the latter smooth situation.

\subsection{Computation of $\mathbf{c_{1,\R}(\cK_{\cX/S}, g^{-1})}$}
Again we assume $S$ to be smooth. Since the notion of the first real Chern form of the relative canonical bundle $(\cK_{\cX/S},g^{-1})$ is functorial, this is sufficient.

The form $\omega^{WP}$ on the total space of a holomorphic family $(f:\cX \to S, \lambda_{\cX/S})$ defines an inner product for horizontal lifts \eqref{eq:vi} of tangent vectors of $S$.
\begin{equation}\label{eq:defphi}
\varphi_{i\ol\jmath}: = \langle v_i, v_j\rangle_{\omega_\cX} = g_{i\ol\jmath} - a^\alpha_i a^\ol\beta_{\ol\jmath} g_{\alpha\ol\beta}.
\end{equation}
We denote the pointwise inner product of harmonic \ks forms $A_i$ and $A_j$ by $A_i \cdot A_\ol\jmath$.
\begin{lemma}
\begin{equation}\label{eq:varphi}
  \Box \varphi_{i\ol\jmath} = -\Theta_{i\ol\jmath} + A_i \cdot A_\ol\jmath.
\end{equation}
\end{lemma}
Note that the tensor $\Theta_{i\ol\jmath}$ only depends upon the parameter $s$ by \eqref{eq:Thetaij}.
\begin{proof}
  \begin{gather*}
   \Box \varphi_{i\ol\jmath} =- g^{\ol\delta\gamma} g_{i\ol\jmath|\ol\delta\gamma}+ g^{\ol\beta\alpha}g^{\ol\delta\gamma} \left(A_{\ol\jmath\alpha;\gamma}\,A_{i\ol\beta;\ol\delta}+ g_{\alpha\ol\jmath;\ol\delta\gamma}\,g_{i\ol\beta}+ g_{\alpha\ol\jmath}\,g_{i\ol\beta;\ol\delta\gamma}+ g_{\alpha\ol\jmath;\ol\delta}\,g_{i\ol\beta;\gamma}     \right)
  \end{gather*}
Because of the Ricci flatness of $\omega_{\cX_s}$, and \eqref{eq:dbclsd} we have $g^{\ol\delta\gamma} g_{\alpha\ol\jmath;\ol\delta\gamma} = g^{\ol\delta\gamma} g_{\alpha\ol\jmath;\gamma\ol\delta} = 0$. Also $g^{\ol\delta\gamma}g_{i\ol\beta;\ol\delta\gamma}=0$ by \eqref{eq:dbclsd}.

Finally
\begin{gather*}
 - g^{\ol\delta\gamma}\, g_{i\ol\jmath|\ol\delta\gamma} = -g^{\ol\delta\gamma} \, g_{\gamma\ol\delta|i\ol\jmath} = -  (g^{\ol\delta\gamma} g_{\gamma\ol\delta|i})_{|\ol\jmath} + g^{\ol\delta\alpha}\,g_{\alpha\ol\beta|\ol\jmath}\,g^{\ol\beta\gamma}\,g_{\gamma\ol\delta|i} \end{gather*}
Now $ (g^{\ol\delta\gamma} g_{\ol\delta\gamma|i})_{|\ol\jmath}= (\log g)_{i\ol\jmath}= \Theta_{i\ol\jmath}$, and the claim follows.
\end{proof}
\begin{proof}[Proof of \eqref{eq:relcan}]
We integrate \eqref{eq:varphi}, and apply \eqref{eq:Thetaij}.
\end{proof}

\begin{proof}[Proof of Theorem~\ref{th}]
We claim that for $\omega_{\cX}$ according to Proposition~\ref{pr:omegaX}, and a suitable number $c>0$ the global form
$$
\omega_\cX + c f^*\omega^{WP}
$$
is \ka. The number $c>0$ will be chosen {\em globally}\/ for the whole given family. Namely we assume that the Green's function $G_s(z,w)$ for the Laplacian of functions on the fibers $\cX_s$ satisfies
$$
G_s(z,w)>-c
$$
for all $s\in S$.

The positive definiteness can be verified after restricting the base space to a neighborhood of any given point $s_0\in S$. Moreover, given a non-zero tangent vector of $S$ at a point $s_0\in S$, we find a subspace $\wt S\subset S$ of embedding dimension one, which is contained in the first infinitesimal neighborhood of $s_0$ in $S$ determined by the given tangent vector. We denote by $\cX_{\wt\cS}= \cX\times_{\cS}\wt S$ the restricted family. As the construction of the \wp metric is functorial, i.e.\ compatible with base change, and all directions are unobstructed, we can use the formulas from the previous sections.

Again $s$ denotes a holomorphic coordinate on $\wt S$, and by abuse of notation the letter $s$ also denotes the corresponding index. Let $\wt\cX  \to \wt S $ be the above family.

  Let $\varphi_{s\ol s}(z,s)$ be given in the sense of \eqref{eq:defphi}. One can see easily that $\varphi_{s\ol s}$ comes from an $(n+1)\times(n+1)$-determinant:
  $$
  \varphi_{s\ol s}\, g  = \det \left(
\begin{array}{cc}
g_{s\ol s} & g_{s\ol\beta}\\ g_{\alpha\ol s}& \gab
\end{array}
\right)
  $$
  so that
  $$
  \varphi_{s\ol s}{\ii} ds\we \ol{ds}  \, g \, dV = \omega^{n+1}_{\cX_{\wt S}}.
  $$
  The normalization \eqref{eq:norm} implies that for any fixed $s\in S$ the harmonic projection of $\varphi_{s\ol s}(z,s)$ vanishes. Denote by $G_s$ the Green's operators for functions on the fiber $\cX_s$ as above. We set $\Box=\Box_s$ for the Laplacians on the fiber, and we denote by $H_s$ the respective harmonic projections. For any function $\chi$ on $\cX_s$, $s\in S$
  $$
  G_s(\chi)= \int_{\cX_s} G_s(z,w)\chi(w)g(w) dV(w)
  $$
  holds.

  Now by \eqref{eq:varphi}, since the harmonic projections $H_s(\varphi_{s\ol s})$ vanish, we have for $s \in \wt S$
  $$
  \varphi_{s\ol s}= G_s( \Box_s (\varphi_{s\ol s})) = G_s(-\Theta_{s\ol s} + A_s \cdot A_{\ol s})= G_s(A_s \cdot A_{\ol s}).
  $$
  The last equality holds because of \eqref{eq:Thetaij}. By our assumption, on all fibers $G_s(z,w) \geq -c$ for some $c>0$ (and uniformly for all $s\in S$).

  Now on $\wt S$ we have $G_s(A_s\cdot A_\ol s)\gequov{eq:relcan} -c \cdot {\rm vol}(\cX_s)\, \Theta_{s\ol s}$.

  We take the following normalization for the \ka\ form on the total space $\cX$. Let
  \begin{equation}\label{eq:ansatzomega}
  \wt\omega_\cX = \omega_\cX + (c+1)\, f^* \omega^{WP}.
  \end{equation}
  Then
  \begin{equation}\label{eq:finestim}
  \wt\omega_\cX^{n+1}|\cX_{\wt S}= (\varphi_{s\ol s} + (c+1) \Theta_{s\ol s}) \, g\, dV \we \ii ds \we \ol{ds}\geq  {\rm vol}(\cX_s)\, \Theta_{s\ol s} g \, dV \we \ii ds \we \ol{ds}.
  \end{equation}
  This holds for any space $\wt S\subset S$ of the above kind through any given point of $S$ and determined by an arbitrary tangent vector. It proves the positive definiteness of $\wt \omega_\cX$.
\end{proof}
We just proved the following fact.
\begin{remark}
  By \eqref{eq:finestim} for the components of type $(n,n)$ on $\cX$ and type $(1,1)$ in horizontal direction
  $$
  \wt\omega^{n+1} \geq \wt\omega^{n}\we f^*\omega^{WP}
  $$
  holds, analogous to \cite[(13)]{sch:inv12}.
\end{remark}

\subsection*{Acknowledgements} The second named author was supported by the National Research Foundation (NRF) of Korea grant funded by the Korea government (No.\ 2018R1C1B3005963) and Pusan National University Research Grant, 2017.


\begin{thebibliography}{Xyz00}

\bibitem[Br15]{br:diss} Braun, Matthias: Positivität relativer kanonischer Bündel und Krümmung höherer direkter Bildgarben auf Familien von Calabi-Yau-Mannigfaltigkeiten, Dissertation, Marburg 2015.

\bibitem[BCS15]{bcs} Braun, Matthias, Choi, Young-Jun, Schumacher, Georg: Positivity of direct images of fiberwise Ricci-flat metrics on Calabi-Yau fibrations. arXiv:1508.00323v4.

\bibitem[Ca57]{calabi}  Calabi, Eugenio: On Kähler manifolds with vanishing canonical class. Algebraic geometry and topology. A symposium in honor of S.~Lefschetz, pp.\ 78--89. Princeton University Press, Princeton, N.J., 1957.

\bibitem[Ch70]{ch} Cheeger, Jeff: A lower bound for the smallest eigenvalue of the Laplacian, in Problems in Analysis (A Symposium in Honor of S.~Bochner), Princeton University Press, Princeton, 1970, pp.\ 195--199.

\bibitem[C-Y81]{c-y} Cheeger, Jeff, Yau, Shin-Tung: A lower bound for the heat
    kernel. Commun.\ Pure Appl.\ Math.\ {\bf 34}, 465--480 (1981).

\bibitem[Cr80]{cr} Croke, Christopher B.: Some isoperimetric inequalities and eigenvalue estimates, Ann.\ Scient Éc. Norm. Sup. {\bf 13},  419--435 (1980).

\bibitem[F-S90]{f-s:extremal} Fujiki, Akira; Schumacher, Georg: The moduli space
    of extremal compact \ka\  manifolds and generalized Weil-Petersson metrics.
    Publ.\ Res.\ Inst.\ Math.\ Sci.\ {\bf 26}, 101--183 (1990).

\bibitem[Gr70]{gri} Griffiths, Philipp A.: Periods of integrals on algebraic manifolds. III: Some global differential-geometric properties of the period mapping. Publ.\ Math. IHES {\bf 38}, 125--180 (1970).

\bibitem[Lu01]{lu} Lu, Zhiqin: On the Hodge metric of the universal deformation space of Calabi-Yau threefolds. J.\ Geom.\ Anal.\ {\bf 11} 103--118 (2001).

\bibitem[L-S04]{lu-sun}  Lu, Zhiqin; Sun Xiaofeng: Weil–Petersson geometry on moduli space of polarized Calabi–Yau manifolds, J. Inst. Math. Jussieu {\bf3}, 185-229 (2004).

\bibitem[N86]{nanni} Nannicini, Antonella: Weil-Petersson metric in the space of compact polarized \ke manifolds of zero first Chern class, Manuscr.\ math.\ {\bf 54}, 405--438 (1986).

\bibitem[R-Z11]{r-z} Rong, Xiaochun; Zhang, Yuguang: Continuity of extremal transitions and flops for Calabi–Yau manifolds. J.\ Diff. Geom.\ {\bf 89}, 233--269 (2011).

\bibitem[Sch85]{Sch0}  Schumacher, Georg: On the geometry of moduli spaces, Manusc.\ math.\ {\bf 50}, 229--267 (1985).

\bibitem[Sch84]{Sch1} Schumacher, Georg: Moduli of polarized K\"ahler manifolds, Math. Ann. {\bf 269}, 137--144 (1984).

\bibitem[Sch08]{sch-preprint08} Schumacher, Georg:  Positivity of relative canonical bundles for families of canonically polarized manifolds, arXiv:0808.3259.

\bibitem[Sch12]{sch:inv12} Schumacher, Georg: Positivity of relative canonical bundles and applications. Invent.\ math. {\bf 190}, 1--56 (2012) and  Erratum to: Positivity  of relative canonical bundles and applications. Invent.\ math.\  {\bf 192}, 253--255 (2013).

\bibitem[Siu86]{siu} Siu, Yum-Tong: Curvature of the Weil-Petersson metric in the moduli space of compact Kähler-Einstein manifolds of negative first Chern class. Contributions to several complex variables, Hon.\ W. Stoll, Proc.\ Conf.\ Complex Analysis, Notre Dame/Indiana 1984, Aspects Math.\ {\bf E9}, 261--298 (1986).

\bibitem[Ta15]{ta} Takayama, Shigeharu: On Moderate Degenerations of Polarized Ricci-Flat Kähler Manifolds. J.\ Math.\ Sci.\ Univ.\ Tokyo {\bf 22}, 469--489 (2015).

\bibitem[Ti86]{tian} Tian, Gang: Smoothness of the universal deformation space of compact Calabi-Yau manifolds and its Petersson-Weil metric, Mathematical aspects of string theory (San Diego, Calif., 1986), World Sci.\ Publishing, Singapore, 1987, Adv.\ Ser.\ Math. Phys., {\bf 1}, 629--646.

\bibitem[To89]{todorov} Todorov, Andrey: The Weil-Petersson geometry of the moduli space of $SU(n \geq 3)$ (Calabi-Yau) manifolds. I, Comm.\ Math.\ Phys., {\bf 126},  325--346 (1989).

\bibitem[T15]{tosatti16} Tosatti, Valentino: Families of Calabi–Yau Manifolds and Canonical Singularities. Int.\ Math.\ Res.\ Not.\ {\bf  20} 10586--10594 (2015).

\bibitem[W03]{wang} Wang, Chin-Lung: Curvature Properties of the Calabi-Yau Moduli. Documenta Mathematica {\bf 8}, 577--590 (2003).

\bibitem[Yau77]{yau} Yau, Shin-Tung: Calabi’s conjecture and some new results in algebraic geometry. Proc.\ Natl.\ Acad.\ Sci.\ USA {\bf 74}, 1798--1799 (1977).

\bibitem[Yau78]{calyau} Yau, Shin-Tung: On the Ricci curvature of a compact Kähler
    manifold and the complex Monge-Ampère equation, Commun.\ Pure Appl.\
     Math.\ {\bf 31}, 339--411 (1978).

\bibitem[Z16]{zhang} Zhang, Yuguang: Completion of the moduli space for polarized Calabi-Yau manifolds. J.\ Diff.\ Geom.\ {\bf 103}, 521--544 (2016).

\end{thebibliography}
\end{document}